\documentclass[10pt]{article}

\usepackage{latexsym}
\usepackage{amssymb}
\usepackage{amscd}
\usepackage{amsmath}
\usepackage{amsthm}
\usepackage[usenames,dvipsnames]{color}
\usepackage{graphicx}
\usepackage{wrapfig}
\usepackage{tikz}
\usetikzlibrary{shapes.geometric}
 \usepackage{url}
\usepackage{fancyvrb}

\usepackage[mathscr]{euscript}
\usepackage{mathtools}

 \usepackage[colorlinks=true,linkcolor=blue,citecolor=blue,urlcolor=magenta,pdfpagelabels=false]{hyperref}

\usepackage[margin=3cm]{geometry}

\usepackage[capitalise]{cleveref}
\usepackage{charter}

\def\eop{\hfill\rule{2.0mm}{2.0mm}}

\def\bR{{{\mathbb R}}}

\def\geq{\geqslant}

\newcommand{\dd}{{\mathrm d}}
\newcommand{\Iso}{{\rm Iso}}

\newtheorem{proposition}{Proposition}[section]
\newtheorem{lemma}[proposition]{Lemma}
\newtheorem{theorem}[proposition]{Theorem}
\newtheorem{corollary}[proposition]{Corollary}

\theoremstyle{definition}

\theoremstyle{remark}

\newcommand*\pFqskip{8mu}
\catcode`,\active
\newcommand*\pFq{\begingroup
        \catcode`\,\active
        \def ,{\mskip\pFqskip\relax}%
        \dopFq
}
\catcode`\,12
\def\dopFq#1#2#3#4#5{%
        {}_{#1}F_{#2}\biggl[\genfrac..{0pt}{}{#3}{#4};#5\biggr]%
        \endgroup
}

\catcode`,\active

\catcode`\,12

\begin{document}
\title{\sc Isoperimetric Ratios of Toroidal Dupin Cyclides}

\date{\today}

\author{
Alin Bostan\thanks{Inria, Univ. Paris-Saclay, France. Email: \href{mailto:alin.bostan@inria.fr}{alin.bostan@inria.fr}. }
\and
 Thomas Yu\thanks{
Department of Mathematics, Drexel University. Email: \href{mailto:yut@drexel.edu}{yut@drexel.edu}. Supported in part by the National Science Foundation grants DMS 1913038.}
\and
Sergey Yurkevich\thanks{A\&R TECH, Austria. Email: \href{mailto:sergey.yurkevich@univie.ac.at}{sergey.yurkevich@univie.ac.at}. Supported in part by \href{https://oead.at/en/}{WTZ collaboration}/\href{https://www.campusfrance.org/}{Amadeus project} FR 02/2024}
}

\makeatletter \@addtoreset{equation}{section} \makeatother
\maketitle

\begin{abstract}
The combination of recent results
due to Yu and Chen [Proc. AMS 150(4), 2020, 1749--1765]
and to Bostan and Yurkevich [Proc. AMS 150(5), 2022, 2131–2136]
shows that the 3-D Euclidean shape of the square Clifford torus is uniquely determined by its isoperimetric ratio. This solves part of the still open uniqueness problem of the Canham model for biomembranes. In this work we investigate the generalization of the aforementioned result to the case of a \emph{rectangular} Clifford torus.
Like the square case, we find closed-form formulas in terms of hypergeometric functions for the isoperimetric ratio of its stereographic projection to $\mathbb{R}^3$ and show that the corresponding function is strictly increasing. But unlike the square case, we
show that the isoperimetric ratio \emph{does not} uniquely determine the Euclidean shape of a rectangular Clifford torus.
\end{abstract}

\vspace{.2in}

\section{Introduction} \label{sec:Intro}
In differential geometry, the \emph{Willmore energy} of a given (smooth and connected) closed surface~$S$ embedded in the 3-D Euclidean space is defined as the integral, taken over~$S$, of the square of the mean curvature.
It measures to which extent such a surface deviates from a round sphere.
Willmore~\cite[Thm.~1]{Willmore65} proved that among all closed surfaces in $\mathbb{R}^3$
the Willmore energy is at least equal to $4 \pi$, and that it is minimized exactly for the round sphere.
However, for surfaces with positive genus a long-standing conjecture due to Willmore~\cite{Willmore65} predicts a better lower bound, namely $2\pi^2$.
In 2014 Marques and Neves \cite{MarquesNeves:Willmore} proved Willmore's conjecture, by showing that among all closed surfaces in $\mathbb{R}^3$ of genus $g \geq 1$ the Willmore energy is minimized exactly for the class of projections to~$\mathbb{R}^3$ of the \emph{square Clifford torus} $C_{\pi/4}$ defined by
\[
    \mathcal{C}_{\pi/4} = \big\{[\cos u, \sin u, \cos v, \sin v] / \sqrt{2}: u,v \in [0,2\pi] \big\} \subseteq \mathbb{S}^3 \subseteq \mathbb{R}^{4}.
\]
Closely connected is the classical \emph{Canham model for biomembranes}~\cite{Canham70} which asks for a minimizer of the Willmore energy for a closed surface $S$ in $\mathbb{R}^3$ with given genus, surface area and volume. By the scale invariance of Willmore energy, % Sergey: I elaborate on the equivalence ...
prescribing the surface area and volume of $S$ is equivalent to fixing the surface's \emph{isoperimetric ratio}, defined as
\begin{eqnarray} \label{eq:IsoDef}
v(S) \coloneqq 6 \sqrt{\pi} \cdot \frac{{\rm Vol}(S)}{{\rm Area}(S)^{3/2}},
\end{eqnarray}
which is normalized in such a way that for the round sphere it is % Sergey: drop 'exactly', more concise
equal to~1.

Motivated by the open \emph{uniqueness problem} of Canham's model (discussed in \cite{YuChen:Uniqueness,KusnerMondinoSchulze:MSRI})
the following result was established as a combination of several works by Chen and the authors of the present paper~\cite{YuChen:UniquenessAMS,BoYu22}, as well as Melczer and Mezzarobba~\cite{MelczerMezzarobba2022}:
\begin{theorem} \label{thm:MainYuChen}
The 3-D Euclidean shape
of $C_{\pi/4}$ is uniquely determined by its isoperimetric ratio.
\end{theorem}
By ``3-D Euclidean shapes" we refer to the different
stereographic projections of $C_{\pi/4}$ into $\bR^3$ for which the center of projection does not lie on the torus.
Two subsets of $\bR^3$ are considered to have the same Euclidean shape (or simply the same \emph{shape}) if they can be transformed
from one to another by a similarity transformation; in the latter, we call two such subsets \emph{homothetic}.
We refer to \cite[Figure 1]{YuChen:UniquenessAMS} for a depiction of various Euclidean shapes of the square Clifford torus.

The stereographic images of ${C}_{\pi/4}$ in $\mathbb{S}^3$ are homothetic to the sphere inversions of the torus of revolution~$T_{\sqrt{2}}$ with major radius $\sqrt{2}$ and minor radius $1$.
More generally, the stereographic images of the
\emph{rectangular Clifford torus}
\begin{eqnarray} \label{eq:CliffordTorusS3}
C_\alpha = \Big\{\big[\cos\alpha \cos u, \cos\alpha \sin u, \sin\alpha \cos v, \sin\alpha \sin v\big]: u,v \in [0,2\pi] \Big\}, \;\; \alpha \in (0,\pi/2),
\end{eqnarray}
are homothetic to  the sphere inversions of $T_{\csc \alpha}$, where % (or \sec(a) ?)
\[
    T_R \coloneqq \big\{\big[ \big(R + \sin v \big)\cos u, \big( R + \sin v \big)\sin u, \; \cos v \big]: u,v \in [0,2\pi] \big\}
\]
is the torus of revolution with major radius $R$ and minor radius $1$.

Denote by $i_\mathbf{x}$ the sphere inversion map about a
unit sphere centered at $\mathbf{x} \in \bR^3$. If $R \in (1,\infty)$ and $\mathbf{x} \notin T_R$, then
$i_\mathbf{x}(T_R)$ is called a
\emph{toroidal Dupin cyclide}. Theorem~\ref{thm:MainYuChen} is equivalent to saying that the Euclidean shapes of
$\{i_\mathbf{x}(T_{\sqrt{2}}): \mathbf{x} \in \bR^3 \backslash T_{\sqrt{2}}\}$
are in one-to-one correspondence with their isoperimetric ratios (\ref{eq:IsoDef}).
Note that for notational convenience, if ${\bf x} \in T_R$, we define $i_{\bf x}(T_R)$ to be an arbitrary round sphere.

The proof of \cref{thm:MainYuChen} crucially relies on the following fact:
\begin{theorem} \label{thm:MeMeBoYu}
The isoperimetric ratio of $i_{[\varrho,0,0]}(T_{\sqrt{2}})$ increases monotonically for $\varrho \in [0, \sqrt{2}-1)$.
\end{theorem}
This statement was conjectured in \cite{YuChen:UniquenessAMS} and independently proved in \cite{BoYu22} and \cite{MelczerMezzarobba2022}. In \cite{BoYu22} \cref{thm:MeMeBoYu} is proved based on hypergeometric representations of the area and volume of $i_{[\varrho,0,0]}(T_{\sqrt{2}})$.
These hypergeometric representations could be found using P-recursive sequences given in \cite[Proposition 4.1]{YuChen:UniquenessAMS}.
In \cite{MelczerMezzarobba2022}, the result is proved with rigorous asymptotic analysis of another but related P-recurrence associated with the derivative of the isoperimetric ratio of $i_{[\varrho,0,0]}(T_{\sqrt{2}})$ with respect to $\varrho$ (also derived in \cite[Prop.~4.1]{YuChen:UniquenessAMS}).

\bigskip

Having established the case of the \emph{square Clifford torus} $C_{\pi/4}$ % (equivalently its trivial projection $T_{\sqrt{2}}$)
in \cref{thm:MainYuChen} and \cref{thm:MeMeBoYu}, in this work we investigate the case of a general \emph{rectangular Clifford torus} $C_\alpha$ for $\alpha \in (0, \pi/2)$. % (equivalently $T_{R}$ for $R>1$).
Along the way, we streamline the analysis in the case of $C_{\pi/4}$ and also
% as well as (Sergey: "as well as" does not work here, it's easiest if we change it to 'and also')
prove the following curious fact:
\begin{proposition} \label{thm:Main2}
    Theorem~\ref{thm:MainYuChen} does \textbf{not} hold for $C_{\alpha}$ with $\alpha \neq \pi/4$.
\end{proposition}
Phrased differently, this theorem says that for any $R > 1$ such that $R \neq \sqrt{2}$, the isoperimetric ratio of $i_{\bf x}(T_R)$ does \emph{not} uniquely determine its Euclidean shape.
While \cref{thm:MainYuChen} does not generalize to a general $\alpha$, \cref{thm:MeMeBoYu} does generalize to a general $R>1$: 

\begin{theorem} \label{thm:Main}
For any $R \in (1,\infty)$, when $\varrho$ increases from $0$ to $R-1$,
the isoperimetric ratio of $i_{[\varrho,0,0]}(T_R)$ increases monotonically from $3/(2\sqrt{\pi R})$ to $1$. % $\varrho \in [0, R-1)$.
\end{theorem}

The proof of \cref{thm:Main} appeared in a slightly modified version in the PhD thesis~\cite[\S4.3]{Yurkevich23} of the third author and is reproduced with some adaptations in this paper. Similarly to the approach in \cite{BoYu22} for the case $R=\sqrt{2}$, the proof technique presented here is based on hypergeometric representations of the area and volume of $i_{[\varrho,0,0]}(T_R)$.
Like in \cite[\S4.3]{Yurkevich23}, but contrary to \cite{YuChen:UniquenessAMS,BoYu22}, we explain here in more details 
how the hypergeometric representations could be found and proved using the technique of creative telescoping.  Moreover, it turns out that when $R\neq \sqrt{2}$ the representations involve the generalized hypergeometric function ${}_3F_2$, instead of just the Gaussian hypergeometric function ${}_2F_1$ as seen in \cite{YuChen:UniquenessAMS}. Finally, as evident in \cref{sec:IsoR} the proof of monotonicity becomes more involved in the $R\neq \sqrt{2}$ case.

\bigskip

\noindent
\textbf{Structure of the paper:} In \cref{sec:dupin}, we state three 
properties of the sphere inversions of $T_R$. From these properties, we also see
\begin{itemize}
\item[(i)] how \cref{thm:MainYuChen} follows from \cref{thm:MeMeBoYu}, whereas
\item[(ii)] the natural generalization of \cref{thm:MeMeBoYu}, namely, \cref{thm:Main} does \emph{not} lead to a generalization of \cref{thm:MainYuChen}, hence proving \cref{thm:Main2}.
\end{itemize}
The first of the aforementioned properties is proven in \cite[Section 2]{YuChen:UniquenessAMS}; the latter two will be proved in \cref{sec:Theorem2and3}.
Finally, \cref{sec:hypergeom} contains the proof of our main \cref{thm:Main}.

\section{Properties of Toroidal Dupin cyclides and \cref{thm:Main2}} \label{sec:dupin}
In this section, we first state three results, \cref{thm:Thm1}-\ref{thm:Thm3}, about toroidal Dupin cyclides. \cref{thm:Thm1} is proved in details
in \cite[Section 2]{YuChen:UniquenessAMS}, however Theorem~\ref{thm:Thm2} and \ref{thm:Thm3} are stated
without proof in \cite[Section 2]{YuChen:UniquenessAMS}. We shall fill in the proofs of the latter two results in \cref{sec:Theorem2and3}.

Denote by $\mathcal{C}(\varrho; R)$ the circle in the $\rho$-$z$ plane with
a diameter connecting $(\varrho,0)$ and $\left( (R^2-1)/\varrho,0 \right)$ (see \cref{fig:Circles}),
and let
\[
\mathcal{T}(\varrho; R) \coloneqq
\left\{ (\rho \cos(\theta), \rho \sin(\theta), z): (\rho,z) \in \mathcal{C}(\varrho; R), \; \theta \in [0,2\pi] \right\}
\]
be the torus
defined by revolving $\mathcal{C}(\varrho; R)$ about the $z$-axis.
By convention, $\mathcal{C}(0; R) = \mathcal{C}(\infty; R)$ is the $z$-axis (of the $\rho$-$z$ plane), so $\mathcal{T}(0; R) = \mathcal{T}(\infty; R)$
is the $z$-axis of the $x$-$y$-$z$ space.

\begin{figure}[ht]
\centerline{
\includegraphics[height=2.3in]{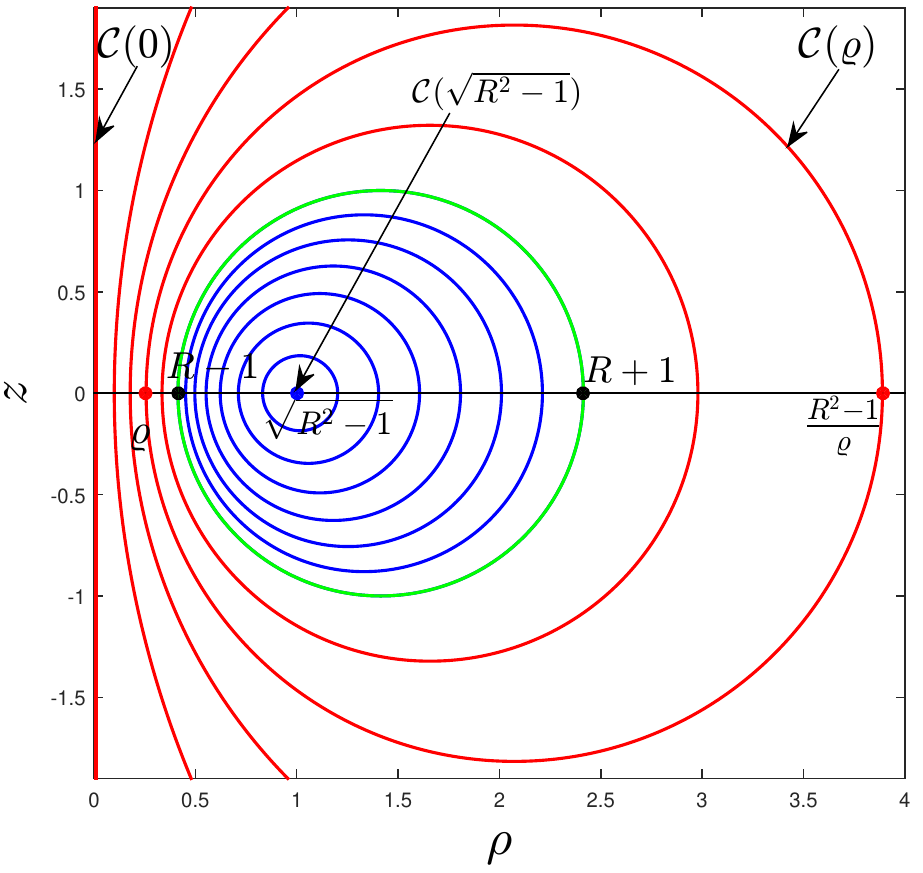}}
\caption{The circles  $\mathcal{C}(\varrho) := \mathcal{C}(\varrho; R)$ for a fixed $R \in (1,\infty)$.
Note that $\mathcal{C}(\varrho; R) = \mathcal{C}( (R^2-1)/\varrho ; R)$
and $\{\mathcal{C}(\varrho; R): \varrho \in [0,\sqrt{R^2-1}]\}$ partitions the half-$\rho$-$z$-plane $\rho\geq 0$, so $\mathcal{T}(\varrho; R) = \mathcal{T}( (R^2-1)/\varrho ; R)$,
and $\{ \mathcal{T}(\varrho; R): \varrho \in [0,\sqrt{R^2-1}]\}$ partitions $\bR^3$.
Notice also that $\mathcal{T}(R\pm 1; R) = T_R$ = the torus of revolution generated by revolving the green circle about the $z$-axis.}
\label{fig:Circles}
\end{figure}

\begin{theorem}[Proven in \cite{YuChen:UniquenessAMS}]\label{thm:Thm1}
    For any fixed $R \in (1,\infty)$ and $\varrho \in \big[ 0, \sqrt{R^2-1} \big]$, all the cyclides in
    \begin{eqnarray} \label{eq:SameShape}
    \big\{i_{\mathbf{x}}(T_R): \mathbf{x} \in \mathcal{T}(\varrho; R) \big\}
    \end{eqnarray}
    are homothetic in $\bR^3$.
\end{theorem}
In virtue of this result, for $\varrho \in [0,\sqrt{R^2-1}]$ we use the shorthand notation $ i_{\varrho}(T_R)$ to represent the (equivalence class of)
common Euclidean shape of the cyclides in \eqref{eq:SameShape}.
Note that $i_{R-1}(T_R)$ is the round sphere.
\begin{theorem}[Stated in \cite{YuChen:UniquenessAMS}, proven in \cref{sec:Theorem2and3}] \label{thm:Thm2}
    For any $R \in (1,\infty)$, $\varrho \in \big[ 0, \sqrt{R^2-1} \big]$,
    \begin{eqnarray}
    \label{eq:rhoDuality}
    i_{\varrho}(T_{R}) = i_{\varrho'}(T_{R'}) \;\;\; \mbox{ where } \;\;\;
    (R', \varrho') = \frac{1}{\sqrt{R^2-1}} \Big(R, \frac{\sqrt{R^2-1} - \varrho}{\sqrt{R^2-1}  + \varrho} \Big).
    \end{eqnarray}
\end{theorem}

\begin{theorem}[Stated in \cite{YuChen:UniquenessAMS}, proven in \cref{sec:Theorem2and3}] \label{thm:Thm3}
    The shapes in
    \begin{eqnarray} \label{eq:All}
    \Big\{ i_{\varrho}(T_{R}) : R \in (1,\infty), \; \varrho \in [0,R-1) \Big\}
    \end{eqnarray}
    are all distinct from each other.
    Moreover, any toroidal Dupin cyclide is homothetic to exactly one of the shapes in~\eqref{eq:All}.
\end{theorem}
Note that for any fixed $R \in (1,\infty)$, the mapping $\rho \mapsto \rho'$ in \eqref{eq:rhoDuality} maps $(R-1,\sqrt{R^2-1}]$ bijectively (and decreasingly)
to $[0,R'-1)$, where $R'=\frac{R}{\sqrt{R^2-1}}$ (as in \eqref{eq:rhoDuality}).
By Theorem~\ref{thm:Thm2},
\begin{eqnarray} \label{eq:DualIntervals}
\Big\{ i_\varrho(T_R): \varrho \in (R-1, \sqrt{R^2-1}] \Big\} =
\Big\{ i_\varrho(T_{R'}): \varrho \in [0, R'-1) \Big\}.
\end{eqnarray}
The special role of the square Clifford torus manifests itself in the $R \leftrightarrow R'$ correspondence:
note that $R'=\sqrt{2}$ if and only if $R=\sqrt{2}$.

Together with Theorem~\ref{thm:Thm1} and \ref{thm:Thm3}, we can conclude:

\begin{corollary} \label{Cor}
    When $R \neq \sqrt{2}$, the set of distinct shapes in $\Big\{i_{\mathbf{x}}(T_R): \mathbf{x} \in \bR^3 \backslash T_R \Big\}$ are
    those in
    \begin{eqnarray} \label{eq:Rectangular1}
    \big\{i_\varrho(T_R): \varrho \in [0, \sqrt{R^2-1}]\backslash\{R-1\} \big\},
    \end{eqnarray}
    or equivalently, those in
    \begin{eqnarray} \label{eq:Rectangular}
    \Big\{ i_\varrho(T_R): \varrho \in [0, R-1) \Big\} \bigcup \Big\{ i_\varrho(T_{R'}): \varrho \in [0, R'-1) \Big\}, \;\;  R'=\frac{R}{\sqrt{R^2-1}}.
    \end{eqnarray}
    When $R=\sqrt{2}$,
    the distinct shapes in $\Big\{i_{\mathbf{x}}(T_{\sqrt{2}}): \mathbf{x} \in \bR^3 \backslash T_{\sqrt{2}} \Big\}$ are those in
    %\begin{eqnarray} \label{eq:Square}
    $\Big\{ i_\varrho(T_{\sqrt{2}}): \varrho \in [0, \sqrt{2}-1) \Big\}$.
\end{corollary}

Now Theorem~\ref{thm:MainYuChen} follows by combining the second part of Corollary~\ref{Cor} with \cref{thm:MeMeBoYu}. Moreover, we can now show that \cref{thm:Main} implies \cref{thm:Main2}:
\begin{proof}[Proof of \cref{thm:Main2}]
By Corollary~\ref{Cor}, when
$R \neq \sqrt{2}$, the distinct shapes of $\{i_{\bf x}(T_R): {\bf x} \in \bR^3\backslash T_R\}$  are those
in \eqref{eq:Rectangular1}.
As $\varrho$ increases from $0$ to $R-1$, the isoperimetric ratio of $i_\varrho(T_R)$ increases from $3/(2\sqrt{\pi R})$ to $1$, by \cref{thm:Main}.
But when $\varrho$ further increases from $R-1$ to $\sqrt{R^2-1}$, the isoperimetric ratio of $i_\varrho(T_R)$
is that of $i_{\rho'}(T_{R'})$, where $(R',\rho')$ given by \eqref{eq:rhoDuality}, which \emph{decreases}
from $1$ to $3/(2\sqrt{\pi R})$. So for any $v \in \big[ \max\big(\frac{3}{2\sqrt{\pi R}}, \frac{3}{2\sqrt{\pi R'}}\big), 1 \big)$,
there are  two distinct $i_\varrho(T_R)$ sharing the same isoperimetric ratio $v$.
\end{proof}
For a visualization of the proof above we refer to Figure~\ref{fig:IsoPlot}; the isoperimetric ratios can be computed either
using numerical integration or the hypergeometric representations
proved in Propositions~\ref{prop:AR} and \ref{prop:VR}.

\begin{figure}[ht]
\centerline{
\includegraphics[height=2.3in]{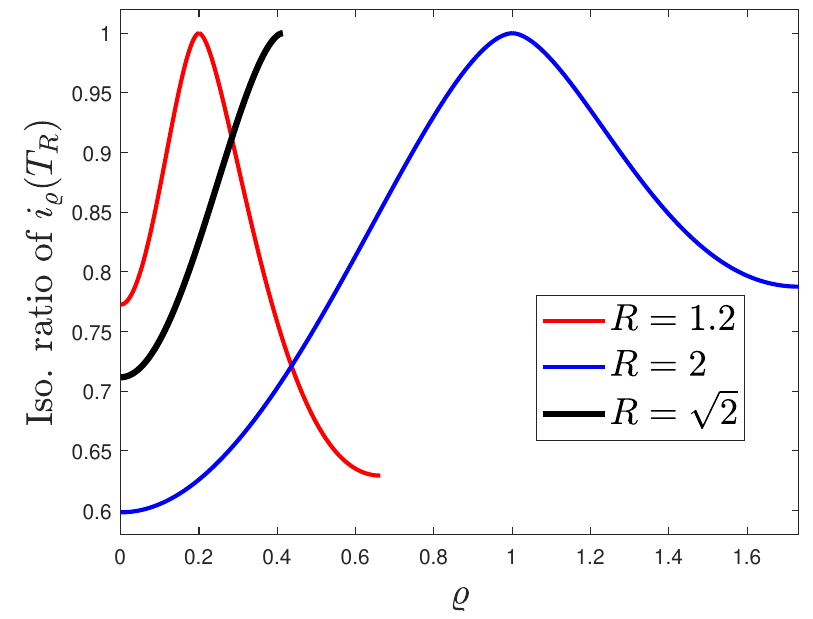}}
\caption{Isoperimetric ratio of $i_\varrho(T_R)$. % for $\varrho \in [0,\sqrt{2}-1]$.
When
$R \neq \sqrt{2}$, all the shapes in
$\{i_\varrho(T_R): \varrho \in [0,\sqrt{R^2-1}]$ are distinct. When $R=\sqrt{2}$, only the shapes in $\{i_\varrho(T_{\sqrt{2}}): \varrho \in [0,\sqrt{2}-1]\}$ are distinct.}
\label{fig:IsoPlot}
\end{figure}

% \vspace{5cm}

\section{Maxwell ratios and proofs of Theorem~\ref{thm:Thm2} and \ref{thm:Thm3}} \label{sec:Theorem2and3}
The proofs of the Theorems~\ref{thm:Thm2} and~\ref{thm:Thm3} boil down to an elementary calculation involving (circle inversions of) circle pairs once we accept the following facts, which follow
from the Maxwell's ellipse-hyperbola characterization of Dupin cyclides  \cite{Maxwell:Cyclide,Boehm:cyclide,CDH:cyclide}.

Every torodial Dupin cyclide $\mathfrak{C}$ has two mutually orthogonal symmetry planes, at one of which the cross section
of $\mathfrak{C}$ consists of \emph{two circles exterior to each other}, and the other consisting of \emph{two circles with one lying inside another}.
We call the former symmetry plane $P_1$, and the latter $P_2$.
Denote by $r_1$, $r_2$, $r_1 \geq r_2$, the radii of the two circles on the $P_1$ cross-section, and by $d$ the distance of their centers, then the ratio
\begin{eqnarray} \label{eq:P1ratio}
r_1:r_2:d
\end{eqnarray}
uniquely determines the Euclidean shape of $\mathfrak{C}$; we refer to \eqref{eq:P1ratio} as the \emph{$P_1$-ratio of $\mathfrak{C}$}.
Similarly, denote by $\tilde{r}_1$, $\tilde{r}_2$, $\tilde{r}_1 \geq \tilde{r}_2$, the radii of the two circles on the $P_2$ cross-section, and $\tilde{d}$ the distance of their centers,
then the \emph{$P_2$-ratio}
\begin{eqnarray} \label{eq:P2ratio}
\tilde{r}_1:\tilde{r}_2:\tilde{d}
\end{eqnarray}
also uniquely determines the Euclidean shape of $\mathfrak{C}$. Using Maxwell's property (see \cite[Section 2]{YuChen:UniquenessAMS} for details),
$(r_1,r_2,d)$ and $(\tilde{r}_1, \tilde{r}_2, \tilde{d})$ are related by the linear isomorphism
$\tilde{r}_1 = \frac{d + (r_1+r_2)}{2}$, $\tilde{r}_2 = \frac{d - (r_1+r_2)}{2}$, $\tilde{d} = r_1 - r_2$,
so each of the two ratios in \eqref{eq:P1ratio}-\eqref{eq:P2ratio} determines the other. For the proof of Theorem~\ref{thm:Thm3},
it is slightly more convenient to use the ratio
\begin{eqnarray}  \label{eq:Maxwellratio}
a:f:(L-a), \;\mbox{where $a = d/2$, $f=(r_1-r_2)/2$, $L=(r_1+r_2+d)/2$,}
\end{eqnarray}
to determine the Euclidean shape of $\mathfrak{C}$. It happens so that $a$, $f$ and $L$ are the relevant measurements in the
Maxwell's characterization of $\mathfrak{C}$ \cite[Section 2]{YuChen:UniquenessAMS}, we simply refer to \eqref{eq:Maxwellratio} as the \emph{Maxwell ratio} of $\mathfrak{C}$.

\begin{lemma}[Equivalent to Lemma 2.3 in \cite{YuChen:UniquenessAMS}]\label{lemma:P1ratios}
Let $R \in (1,\infty)$ and $\mathfrak{C} = i_{[\varrho,0,0]}(T_R)$.
  \begin{itemize}
  \item[(i)] If $\varrho \in [0,R-1)$,
  the $P_1$-ratio of $\mathfrak{C}$ is
  \begin{eqnarray}
  \label{eq:ratioOutside}
  \begin{aligned}
  r_1:r_2:d = (R + \varrho)^2 - 1: (R - \varrho)^2 - 1 : 2 R (R^2 - \varrho^2 - 1).
   \end{aligned}
   \end{eqnarray}
  \item[(ii)] If $\varrho \in \left(R-1, \sqrt{R^2-1} \right]$,
the $P_1$-ratio of $\mathfrak{C}$ is
   \begin{eqnarray}
  \label{eq:ratioInside}
  \begin{aligned}
r_1:r_2:d = (R-1)((R+1)^2-\varrho^2) : (R+1)(\varrho^2-(R-1)^2) : 4R\varrho.
\end{aligned}
   \end{eqnarray}
\end{itemize}
\end{lemma}
This result is equivalent to \cite[Lemma 2.3]{YuChen:UniquenessAMS}, except that the $P_1$-ratios are presented differently. We recall its proof here with the help of a graphical illustration.
The key point is that, thanks to Theorem~\ref{thm:Thm1}, we only need to consider sphere inversions centered at the $x$-axis. This greatly simplifies the computation of the $P_1$-ratio of $\mathfrak{C}$.

\begin{proof}
In case $(i)$ of the lemma, the $P_1$ symmetry plane of both $\mathfrak{C}$ and $T_R$ is the $x$-$z$ plane,
so the $P_1$ cross-section of $\mathfrak{C}$ is the circle inversion of the $P_1$ cross-section of $T_R$ about the unit circle centered at $(\varrho,0)$ on the $x$-$z$ plane, see the first two panels of \cref{fig:DualCircles}.
In case $(ii)$, the $P_1$ symmetry plane
of $\mathfrak{C}$ is the $x$-$y$ plane, which is the $P_2$ symmetry plane of $T_R$, so
the $P_1$ cross-section of $\mathfrak{C}$ is the circle inversion of the $P_2$ cross-section of $T_R$ about the unit circle centered at $(\varrho,0)$ on the $x$-$y$ plane, as in the next two panels of \cref{fig:DualCircles}.

Furthermore, the circle pairs in these cross-sections all share the same symmetry line (the middle black lines in \cref{fig:DualCircles}), so
their radii and centers
can be calculated easily by 1-D inversions, defined by $i(x,\varrho) \coloneqq \varrho + 1/(x-\varrho)$.
\begin{itemize}
\item[(i)] When $\varrho \in [0,R-1)$,
$r_1 = i(R-1,\varrho)-i(R+1,\varrho))/2$, $r_2 = i(-(R+1),\varrho)-i(-(R-1),\varrho))/2$,
$d=1/2(i(R-1,\varrho)+i(R+1,\varrho)) - 1/2(i(-(R+1),\varrho)+i(-(R-1),\varrho))$.
\item[(ii)] When $\varrho \in (R-1,\sqrt{R^2-1}]$,
$r_1 = i(-(R-1),\varrho)-i(R-1,\varrho))/2$, $r_2 = i(R+1,\varrho)-i(-(R+1),\varrho))/2$,
$d=1/2(i(R+1,\varrho)+i(-(R+1),\varrho)) - 1/2(i(-(R-1),\varrho)+i(R-1,\varrho))$.
\end{itemize}
The $P_1$ ratios then follow by routine calculations.
\end{proof}

\begin{figure}[ht]
\centerline{
\begin{tabular}{cc}
\includegraphics[height=1in]{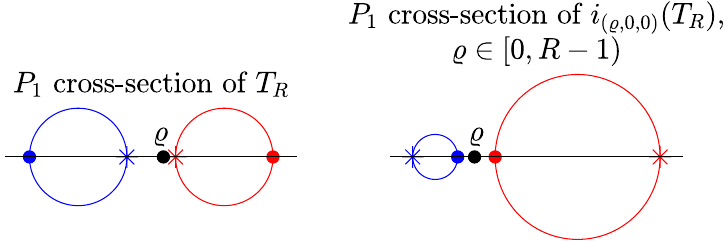} & \hspace{-.2cm}\includegraphics[height=1.05in]{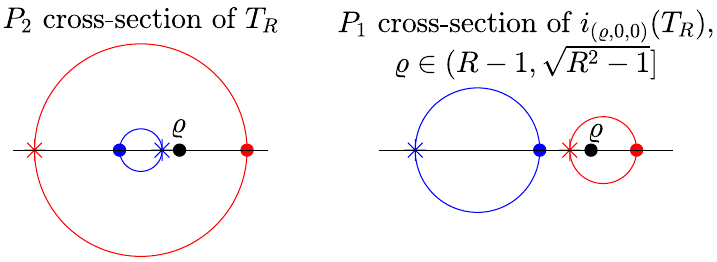}
\end{tabular}}
\caption{The shape space of 3-D torodial Dupin cyclides reduces to 2-D (circle) inversions of circle pairs (which are
 cross-sections of the cyclides), which further reduces to
 calculations involving 1-D inversions.}
\label{fig:DualCircles}
\end{figure}

\begin{proof}[Proof of \cref{thm:Thm2}]
The proof of \cref{thm:Thm2} follows by checking that the ratio in \eqref{eq:ratioInside} with $R$ replaced by $R/\sqrt{R^2-1}$
and $\varrho$ replaced by $\frac{1}{\sqrt{R^2-1}} \frac{\sqrt{R^2-1}-\varrho}{\sqrt{R^2-1}+\varrho}$
is the same as the ratio in \eqref{eq:ratioOutside}.
\end{proof}

%\begin{remark}
%It may worth noting that $T_R$ and $T_{R'}$ ($R'=1/\sqrt{R^2-1}$) correspond to
%\eqref{eq:CliffordTorusS3} with $\alpha = \csc^{-1}(R)$ and $\alpha' = \pi/2 - \alpha$. So up in $\mathbb{S}^3 \subset \bR^4$, the two tori are
%related by the inversion
%$\left[\begin{smallmatrix*}[c] 0 & I_2 \\ I_2 & 0\end{smallmatrix*} \right].$ It is possible to prove Theorem~\ref{thm:Thm2} from this point of view,
%however we did not pursue this approach because it seemed harder to prove Theorems~\ref{thm:Thm1} and \ref{thm:Thm3} this way.
%\end{remark}

\begin{proof}[Proof of \cref{thm:Thm3}] The second claim in the theorem follows from the first claim and Theorem~\ref{thm:Thm1}.
The first claim is equivalent to saying that the $P_1$-ratio in \eqref{eq:ratioOutside} is distinct for distinct $(R,\varrho)$ in
$\mathfrak{T}=\{(R,\varrho): R \in (1,\infty), \varrho \in [0,R-1)\}$, which is equivalent to having distinct Maxwell ratio \eqref{eq:Maxwellratio}.
The latter can be expressed as
\begin{align*}
a:f:(L-a) = d:(r_1-r_2):(r_1+r_2) &= 1: \frac{(R+\varrho)^2-(R-\varrho)^2}{2R(R^2-\varrho^2-1)} : \frac{(R+\varrho)^2+(R-\varrho)^2}{2R(R^2-\varrho^2-1)} \\
&= 1: \frac{2\varrho}{R^2-\varrho^2-1} : \frac{R^2+\varrho^2-1}{R(R^2-\varrho^2-1)}.
\end{align*}
The claim then follows because the map
$(R,\varrho) \stackrel{\phi}{\mapsto} \left(\frac{2\varrho}{R^2-\varrho^2-1}, \frac{R^2+\varrho^2-1}{R(R^2-\varrho^2-1)}\right)$ is injective on  $\mathfrak{T}$.
It is easy to check that
$\phi(R,\varrho)
%\Big( \frac{2\varrho}{R^2-\varrho^2-1}, \frac{R^2+\varrho^2-1}{R(R^2-\varrho^2-1)} \Big)
\in \mathfrak{T}':=\{(a,b): a \in [0,1), b \in (a,1)\}$ for $(R,\varrho) \in \mathfrak{T}$.
One can then see that $\phi: \mathfrak{T} \rightarrow \mathfrak{T}'$ is bijective by verifying
that
$\phi^{-1}(a,b) =
\big( \sqrt{1-a^2}, (b\sqrt{1-a^2}-1)/a \big)/\sqrt{b^2-a^2}
$.
\end{proof}

\section{Hypergeometric representations and proof of Theorem~\ref{thm:Main}}
To prove Theorem~\ref{thm:Main}, we first show that the surface area and enclosing volume of
$i_{[\varrho,0,0]}(T_R)$, for
$\varrho \in [0,R-1)$, can be expressed as integrals of trigonometric functions.
Write
\[
\mathbf{x}(u,v,r) \coloneqq \left[ \big(R + r\sin(v) \big)\cos(u), \;
\big( R + r\sin(v) \big)\sin(u), \;
r\cos(v) \right].
\]
Since the conformal factor of $i_{\bf a}$ at ${\bf x} \in \bR^3\backslash\{\mathbf{a}\}$ is
$$
\lambda^2(\mathbf{a},\mathbf{x}) = \frac{1}{(\langle \mathbf{a},\mathbf{a} \rangle -2\langle \mathbf{a}, \mathbf{x}\rangle + \langle \mathbf{x}, \mathbf{x}\rangle)^2},
$$ i.e. $\langle d i_\mathbf{a}|_\mathbf{x} v, d i_\mathbf{a}|_\mathbf{x} w \rangle = \lambda^2(\mathbf{a},\mathbf{x}) \langle v,w \rangle$,
the area and enclosing volume of $i_{[\varrho,0,0]}(T_{R})$
% and the volume of the solid enclosed by $S_{[a,0,0]}(T_{\sqrt{2}})$,
% denoted by ${\rm solid}(T_{\sqrt{2}})$ below,
are given by
\begin{eqnarray} \label{eq:AreaVolume}
A_R(\varrho) =
\int_0^{2\pi} \int_0^{2\pi} \frac{\dd {\rm Area}(u,v)}{Q(\varrho; \mathbf{x}(u,v,1))^2}, \;\;
\quad
V_R(\varrho) =
\int_0^1 \int_0^{2\pi} \int_0^{2\pi} \frac{\dd {\rm Vol}(u,v,r)}{Q(\varrho; \mathbf{x}(u,v,r))^3},
\end{eqnarray}
where
\begin{equation*}
Q(\varrho; \mathbf{x}) := \frac{1}{\lambda([\varrho,0,0],{\bf x})} =  \varrho^2 - 2 \mathbf{x}_1 \varrho  + \|\mathbf{x}\|^2,
\end{equation*}
$$
\dd{\rm Area}(u,v) = \| \mathbf{x}_u(u,v,1) \times \mathbf{x}_v(u,v,1) \| \, \dd u\, \dd v =  (R+\sin(v)) \, \dd u\, \dd v, % \;\;
$$
$$
\dd{\rm Vol}(u,v,r) = \big| \det[\mathbf{x}_u, \mathbf{x}_v, \mathbf{x}_r] \big| \, \dd u\, \dd v\, \dd r = r(R + r\sin(v)) \, \dd u\, \dd v\, \dd r.
$$
Notice also that
$$\langle \mathbf{x}, \mathbf{x} \rangle  = \| \mathbf{x} \|^2
= R^2 + r^2 + 2 R r \sin(v).$$
It can be shown that $A_R$ and $V_R$ (as in \cite[\S 4.2]{YuChen:UniquenessAMS})
extend to holomorphic functions
on the disk
$$
D := \big\{z \in \mathbb C: |z| < R-1 \big\}.
$$
% So we write $A_R(z)$ and $V_R(z)$ instead of $A_R(\varrho)$ and $V_R(\varrho)$.

Theorem~\ref{thm:Main} is equivalent to showing that the
isoperimetric ratio of $i_{[\varrho,0,0]}(T_R)$, which equals to
\begin{equation*} \label{eq:IsoR}
6 \sqrt{\pi} \cdot \frac{V_R(\varrho)}{A_R(\varrho)^{3/2}} \eqqcolon \Iso_R(\varrho),
\end{equation*}
defines a monotonically increasing function in $\varrho$ from $[0,R-1)$ to $[3/\sqrt{4 \pi R}, 1)$.
This will be proven in \cref{sec:IsoR}, after we establish the hypergeometric expressions for $A_R(\varrho)$ and $V_R(\varrho)$ in \cref{sec:hypergeom}.

\subsection{Hypergeometric representation of $A_R$ and $V_R$} \label{sec:hypergeom}%  and isoperimetric ratio}
The following section is analogous to a part of the PhD thesis of the third author \cite[\S4.3]{Yurkevich23}, however the defining integrals in (\ref{eq:defAR}) and (\ref{eq:defVR}) are slightly different from \cite[(4.8) \& (4.9)]{Yurkevich23}. Practically this almost only results in the substitution $z \mapsto z (R^2-1)^2$, however for sake of completeness we redo all proofs in the new setting. For more details than provided here as well as more explanations on how the expressions were actually \emph{found} (not only proven) we refer to~\cite[\S4.3]{Yurkevich23}.

Let us first write out the integrals for $A_R$ and $V_R$ explicitly:
\begin{align}
A_R(z) &=
\int_0^{2\pi} \int_0^{2\pi} \frac{(R+\sin(v)) \, \dd u\, \dd v}{(z^2 - 2 (R+\sin(v))\cos(u) z + R^2+1+2R\sin(v) )^2}, \label{eq:defAR} \\
V_R(z) &=
\int_0^1 \int_0^{2\pi} \int_0^{2\pi} \frac{r(R+r\sin(v)) \, \dd u\, \dd v\, \dd r}{(z^2 - 2z(R+r\sin(v))\cos(u) + R^2+r^2+2Rr\sin(v))^3}. \label{eq:defVR}
\end{align}

\begin{proposition} \label{prop:AR}
    For $R>1$ and $z \in [0,R-1)$ it holds that
    \begin{align} \label{eq:Aid}
        A_R(z) & = \frac{4 \pi^2 R \left((R^{2}-1)^2-z^4\right)}{\left((R-1)^2 -z^2 \right)^{2} \left((R+1)^2 -z^2 \right)^{2}} \cdot \pFq{2}{1}{-\frac12,-\frac12}{1}{\frac{4 z^2}{\left(R^{2}-1-z^2\right)^{2}}} \\
        & \qquad = \frac{4 \pi^{2} R }{(R^2-1)^2}+\frac{4 \pi^{2} R \left(4 R^{2}+5  \right) }{(R^2-1)^4}z^{2} + \frac{9 \pi^{2} R \left( 4R^4 + 16 R^2 + 5 \right) }{(R^2-1)^6}z^{4} + \cdots. \nonumber
    \end{align}
\end{proposition}
\begin{proof}
    With a Weierstrass substitution we rewrite the integral representation of $A_R(z)$ as an integral of a rational function over the closed surface $\gamma = \{x,y \in \mathbb{C}: |x|=|y|=1\} \subseteq \mathbb{C}^2$:
    \[
        A_R(z) = \frac{1}{2 \pi^2 R} \int_{|x|=|y|=1} \frac{\left(2 R y -y^{2}-1\right) x \; \dd x \dd y}{\left(2 z^{2} y x -\left(2 R y -y^{2}-1\right) \left(x^{2}+1\right) z +2 x \left(R y -1\right) \left(R - y \right) \right)^{2}}.
    \]
    Let $a \in \mathbb{Q}(x,y,z,R)$ be this integrand. Koutschan's improved and implemented version~\cite{Koutschan10} of Chyzak's algorithm for creative telescoping~\cite{Chyzak00} finds a second-order differential operator $L_a \in \mathbb{Q}(z,R)\langle \partial_z \rangle$ and (explicit but huge) rational functions $C_{1},C_2 \in \mathbb{Q}(z,R,x,y)$, such that
    \[
        L_a \cdot a = \partial_x C_1 + \partial_y C_2.
    \]
    We check that the denominators of $C_{1},C_2$ are factors of
    \begin{align*}
        \mathrm{denom}(a) \cdot x \cdot y \cdot (1 + 2 R y - y^2) \cdot  H(z,R,r),
    \end{align*}
    for some polynomial $H(z,R,r) \in \mathbb{Q}[z,R,r]$, i.e. $C_1,C_2$ are well-defined on $\gamma \subseteq \mathbb{C}^2$. It follows that
    \[
        L_a \cdot A_R(z) = \frac{1}{2 \pi^2 R} \int_{|x|=|y|=1} L_a \cdot a \; \dd x \dd y = 0,
    \]
    because $L_a$ commutes with $\oint_\gamma \! \cdot \dd x \dd y$ and $\oint_\gamma \! \partial_x C_1 \dd x \dd y = \oint_\gamma \! \partial_y C_2 \dd x \dd y = 0$.

    Now that we have found a linear differential operator that provably annihilates $A_R(z)$ it is easy to conclude the proof.
    First, we observe using the defining differential equation of the Gaussian hypergeometric function and closure properties of linear ODEs implemented in any computer algebra software that $L_a$ also annihilates the right-hand side of (\ref{eq:Aid}). Moreover, we compute the linear recurrence satisfied by the coefficient sequence $(u_n(R))_{n \geq 0}$ of any power series solution of $L_a \cdot y = 0$:
    \[
        p_{10}(n,R) u_{n+10}(R) + \cdots + p_0(n,R) u_n(R) = 0, \quad \text{ for some } p_0(n,R),\dots,p_{10}(n,R) \in \mathbb{Z}[n,R],
    \]
    and observe that $p_{10}(n,R) = (n + 10)^2 \neq 0$ for $n \in \mathbb{N}$. It follows that any power series solution of $L_a \cdot y = 0$ is uniquely determined by its first $10$ coefficients. Finally, we check that the first $10$ coefficients of both sides of (\ref{eq:Aid}) are identical.
\end{proof}
\begin{proposition} \label{prop:VR}
    For $R>1$ and $z \in [0,R-1)$ it holds that
    \begin{align}\label{eq:Vid}
        V_R(z) & = \frac{2 R \,\pi^{2} \left(R^{2}-1-z^2\right)^{3} }{\left((R-1)^2 -z^2 \right)^{3} \left((R+1)^2 -z^2 \right)^{3}} \cdot \pFq{3}{2}{-\frac32,-\frac32,\frac{3}{2 R^{2}-4}+1}{1, \frac{3}{2 R^{2}-4}}{\frac{4 z^2}{\left(R^{2}-1-z^2\right)^{2}}}\\
        & \qquad = \frac{2 R \,\pi^{2}}{\left(R^2 -1\right)^{3}} +\frac{6 R \,\pi^{2} \left(3 R^{2}+2\right)}{\left(R^{2}-1\right)^{5}}z^2+\frac{3 R \,\pi^{2} \left(48 R^{4}+104 R^{2}+23\right) }{2 \left(R^{2}-1\right)^{7}}z^{4} + \cdots. \nonumber
    \end{align}
\end{proposition}
\begin{proof}
    The proof works analogously to the proof of \cref{prop:AR} with one difference: When creative telescoping is applied to the rational function $v \in \mathbb{Q}(x,y,z,r,R)$ whose triple integral defines $V_R(z)$, the algorithm finds some rational functions $C_{1},C_2,C_3 \in \mathbb{Q}(z,r,R,x,y)$ such that
    \[
        v = \partial_x C_1 + \partial_y C_2 + \partial_r C_3.
    \]
    However, in this case one cannot conclude that $V_R(z) = \int_0^1 \int_\gamma v \dd x \dd y \dd r = {C_3|}_{_{r=1}} - {C_3|}_{_{r=0}}$ because the integration contour intersects the common denominator of $C_1,C_2,C_3$, and therefore integration and differentiation cannot be exchanged. This issue is known as ``irregular'' certificate in the field of creative telescoping.

    We overcome this problem by first proving that
    \begin{align} \label{eq:VRr0}
        &V_R^r(z) \coloneqq \int_0^{2\pi} \int_0^{2\pi} \frac{r(R+r\sin(v)) \, \dd u\, \dd v\, \dd r}{(z^2 - 2z(R+r\sin(v))\cos(u) + R^2+r^2+2Rr\sin(v)))^3} = \nonumber \\
        & \hspace{4cm} Q_1(z,R,r) \pFq{2}{1}{1/2, 1/2}{2}{Z} + Q_2(z,R,r)\pFq{2}{1}{3/2, 3/2}{3}{Z},
    \end{align}
    for $Z = 4 r^{2} z^{2}/(\left(R -r +z \right) \left(R +r -z \right) \left(R +r +z \right) \left(R -r -z \right))$
    and explicit rational functions $Q_1,Q_2 \in \mathbb{Q}(z,R,r)$. Then one can find an explicit antiderivative of the right-hand side in terms of a hypergeometric function and simply evaluate it at $r=1$ and $r=0$:
    \begin{align} \label{eq:VRr}
        V_R(z) = \int_0^1 V_R^r(z) \dd r =
        \frac{2 r^2 R \,\pi^{2} \left(R^{2}-z^2-1\right)^{3} }{\left((R-r)^2 -z^2 \right)^{3} \left((R+r)^2 -z^2 \right)^{3}} \cdot \pFq{3}{2}{-\frac32,-\frac32,\frac{3 r^{2}}{2 R^{2}-4 r^{2}}+1}{1, \frac{3 r^{2}}{2 R^{2}-4 r^{2}}}{\frac{4 r^2 z^2}{\left(R^{2}-r^2-z^2\right)^{2}}} \Bigg|_0^1.
    \end{align}
    Once the expression in (\ref{eq:VRr}) is found, the verification of the proof is easy: Like in the proof of \cref{prop:AR} compute with creative telescoping the annihilator for $V_R^r(z)$ (check that its poles have no intersection with $\gamma$) and verify that the derivative of the right-hand side of (\ref{eq:VRr0}) satisfies the same differential equation. By uniqueness (after checking enough initial conditions) one concludes (\ref{eq:VRr0}) and consequently (\ref{eq:VRr}). We refer to \cite[\S4.3.2]{Yurkevich23} for the explanation on how the expression in (\ref{eq:VRr}) can be first guessed before it is proven.
\end{proof}

\subsection{$\Iso_R$ is increasing} \label{sec:IsoR}
Combining the definition \eqref{eq:IsoR} of $\Iso_R$ with the propositions in \cref{sec:hypergeom} we arrive at
\[
    \Iso^2_R(z) = 36 \pi \cdot \frac{V_R(z)^2}{A_R(z)^{3}} =
    \frac{9}{4 \pi R} \frac{\pFq{3}{2}{-\frac32,-\frac32,\frac{3}{2 R^{2}-4}+1}{1, \frac{3}{2 R^{2}-4}}{\frac{4 z^2}{\left(R^{2}-1-z^2\right)^{2}}}^2}{\pFq{2}{1}{-\frac12,-\frac12}{1}{\frac{4 z^2}{\left(R^{2}-1-z^2\right)^{2}}}^3}  \left(\frac{R^{2}-z^2-1}{R^{2}+z^2-1} \right)^3.
\]
Note that
\begin{align*}
\pFq{3}{2}{-\frac{3}{2},-\frac{3}{2},\frac{3}{2(R^2-2)}+1}{1,\frac{3}{2(R^2-2)}}{1} = \frac{16R^2}{3\pi} \quad \text{and} \quad
\pFq{2}{1}{-\frac{1}{2},-\frac{1}{2}}{1}{1} = \frac{4}{\pi},
\end{align*}
where the first identity is due to
\begin{align} \label{eq:3F2to2F1}
    \pFq{3}{2}{-\frac{3}{2},-\frac{3}{2},\frac{3}{2(R^2-2)}+1}{1,\frac{3}{2(R^2-2)}}{x} = \pFq{2}{1}{-\frac{3}{2},-\frac{3}{2}}{1}{x} + \frac{3}{2} \left( R^2 -2 \right)\cdot x \cdot \pFq{2}{1}{-\frac{1}{2},-\frac{1}{2}}{2}{x}.
\end{align}
We mention that this identity (\ref{eq:3F2to2F1}) also highlights the special role of $R = \sqrt{2}$ covered in \cite{YuChen:UniquenessAMS,BoYu22}, as in this case the second summand vanishes.

From the formulas above we immediately obtain that
\(
\lim_{z \to R-1} \Iso_R(z) = 1
\).
We are now in a position to finish the proof of Theorem~\ref{thm:Main}.
After the substitution
\[
x = \frac{4 z^2}{\left(R^{2}-1-z^2\right)^{2}},
\]
it is enough to show that the function
\[
h_R(x) \coloneqq \frac{\pFq{3}{2}{-\frac{3}{2},-\frac{3}{2},\frac{3}{2(R^2-2)}+1}{1,\frac{3}{2(R^2-2)}}{x}^2}{\pFq{2}{1}{-\frac{1}{2},-\frac{1}{2}}{1}{x}^3} \cdot (1 + (R^2-1)\cdot x)^{-3/2}
\]
is increasing on $x \in (0,1)$ for all $R>1$. Now define the two functions
\begin{align*}
    f(x) \coloneqq \frac{(x+1)^{1/2}}{\pFq{2}{1}{-\frac{1}{2},-\frac{1}{2}}{1}{x}} \quad \text{and} \quad g_R(x) \coloneqq \frac{\pFq{3}{2}{-\frac{3}{2},-\frac{3}{2},\frac{3}{2(R^2-2)}+1}{1,\frac{3}{2(R^2-2)}}{x}}{(1+x)^{3/4} \cdot (1+(R^2-1)\cdot x)^{3/4}}
\end{align*}
and note that clearly $h_R = f^3 \cdot g_R^2$. We will show that both $f$ and $g_R$ increase on $x \in (0,1)$ for any $R>1$. Since $f$ and $g_R$ are positive on this interval, this will be enough to conclude Theorem~\ref{thm:Main}.

The fact that $f(x)$ is increasing on $(0,1)$ is already proved in \cite[Proposition 1]{BoYu22}. For $g_R(x)$ we have to adapt the argument slightly, since the parameter $R$ indeed complicates matters. Let
\[
\frac{4 \cdot g_R'(x) \cdot (1+x)^{7/4} \cdot (1+(R^2-1)\cdot x)^{7/4}}{3 \cdot (1-x)^2 \cdot (R^2-1)} \eqqcolon \sum_{n\geq 0} u_n(R) x^n,
\]
for some rational functions $u_n(R)\in \mathbb{Q}(R)$, $n\geq 1$ and $u_0(R) = 1$. It turns out that the sequence $(u_n(R))_{n\geq0}$ is hypergeometric, and this will allow us to prove that $u_n(R) > 0$ for $R>1$. Clearly, this implies that $g_R'(x)>0$, i.e. $g_R(x)$ is increasing.
For all $n\geq0$ define the sequence of polynomials
\[
p_n(R) \coloneqq 4(R^4 + 4R^2 - 4)n^3 + 6(R^4 + R^2 - 2)n^2 + (2R^4 - 13R^2 + 10)n - 3R^2 + 3.
\]
Then it is not difficult to check with a computer that for $n\geq0$ we have
\[
\frac{u_{n+1}(R)}{u_n(R)} = \frac{(2n-1)(2n+1)\cdot p_{n+1}(R)}{4(n+2)(n+1) \cdot p_n(R)}.
\]
Hence, it holds that $u_{n}(R)>0$ for $R>1$ if we can prove that $p_n(R)>0$ for all $n \geq 1$. Observe that for $n\geq1$ and $R>1$:
\[
p_n(R) > 4 R^4 n^3 + (10-11R^2) n + 3 - 3R^2.
\]
The polynomial on the right-hand side is strictly increasing in $n$ for $n\geq1$ since the larger root of its derivative is $\sqrt{36R^2-33}/(6R^2)<1$. Therefore we can set $n=1$ and obtain
\[
p_n(R) > p_1(R) = (2R^2 - 14/4)^2 + 3/4 > 0.
\]
The proof of Theorem~\ref{thm:Main} is now completed. \eop

\bibliographystyle{abbrv}
{\small
\bibliography{main}
% \bibliography{refinement}
}
\end{document}